\documentclass[12pt]{amsart}

\usepackage{amssymb, hyperref}

\advance\hoffset by -1truein
\advance\voffset by -.5truein

\newtheorem{thm}{Theorem}
\newtheorem{lem}{Lemma}
\newtheorem{cor}{Corollary}
\newtheorem{prop}{Proposition}
\theoremstyle{definition}
\newtheorem{defn}{Definition}
\newtheorem{rem}{Remark}
\newtheorem{exmp}{Example}

\def\Kerr{\mathop{\fam 0 Ker}\nolimits}
\def\Curr{\mathop{\fam 0 Cur}\nolimits}
\def\Cend{\mathop{\fam 0 Cend}\nolimits}
\def\End{\mathop{\fam 0 End}\nolimits}
\def\ConfAs{\mathop{\fam 0 ConfAs}\nolimits}
\def\Span{\mathop{\fam 0 Span}\nolimits}
\def\oo#1{\mathbin{{}_{(#1)}}}

\title{On finite representations of conformal algebras}

\author{Pavel Kolesnikov}

\thanks{Partially supported by MD-2438.2009.1,
RFBR 09-01-00157.}

\address{Sobolev Institute of Mathematics, Novosibirsk, Russia}
\address{University of California, San Diego, USA}

\begin{document}

\begin{abstract}
We prove a finite torsion-free associative conformal algebra
to have a finite faithful conformal representation.
As a corollary, it is shown that one may join a conformal
unit to such an algebra. Some examples
are stated to demonstrate that a conformal unit can not be joined to any torsion-free
associative conformal algebra. In particular, there exist associative conformal
algebras of linear growth and even locally finite ones that have no finite faithful
representation.
We also consider the problem of existence of a finite faithful representation
for a torsion-free finite Lie conformal algebra (the analogue of Ado's Theorem).
It turns out that the conformal analogue of the Poincar\'e---Birkhoff---Witt Theorem
would imply the Ado Theorem for finite Lie conformal algebras. We also prove that
every torsion-free finite solvable Lie conformal algebra has a finite faithful
representation.
\end{abstract}

\maketitle

\section{Introduction}

Conformal algebras
were introduced in \cite{Kac1996} as a useful tool for studying vertex algebras,
see \cite{FLM1998} as a general reference.
The structure of a (Lie) conformal algebra
encodes the singular part of the
operator product expansion (OPE) which is responsible for the
commutator of two fields.
Namely, suppose $(V, Y, \mathbf 1)$ is a vertex algebra, where $V$ is
a space of states, $Y:V\to \End V[[z,z^{-1}]]$ is a space-field
correspondence, $\mathbf 1\in V$ is a vacuum vector. The
commutator of two fields can be expressed as a finite distribution
\[
[Y(a,w),Y(b,z)] = \sum\limits_{n\ge 0} \dfrac{1}{n!} Y(c_n, z)
 \dfrac{\partial^n\delta(w-z)}{\partial z^n},\quad a,b\in V,
\]
where $c_n\in V$,
$\delta(w-z)=\sum\limits_{m\in \mathbb Z} w^m z^{-m-1}$
is the formal delta-function.
Then
\[
Y(c_n, z) = \mathrm{Res}_{w=0} (w-z)^n[Y(a,w),Y(b,z)],
\]
where $\mathrm{Res}_{w=0}F(w,z) $ means the formal residue at
$w=0$, i.e., the formal series in $z$ that is
a coefficient of $F(w,z)$ at~ $w^{-1}$.

One may consider the
space of fields $\{Y(a,z)\mid a\in V\}$
as an algebraic system with operations
$D(\cdot )$ and $(\cdot \oo{n}\cdot )$, $n\ge 0$,
where $D$ is the formal derivation with respect to $z$,
$Y(a,z)\oo{n} Y(b,z) = Y(c_n,z)$. The system obtained is known
as a (Lie) conformal algebra, it has precise axiomatic
description. Roughly speaking, Lie conformal algebras relate
to vertex algebras as ordinary Lie algebras relate to their
associative enveloping algebras.

Conformal algebras
as well as their representations and cohomologies
have been studied in a series of papers, see, e.g.,
\cite{BKV1999, DK1998, CK1997, Roitman1999, Roitman2000,
BKL2003, Retakh2001, Retakh2006}.
In particular, associative conformal algebras naturally appear
in the study of representations of Lie conformal algebras.

From the algebraic point of view, all these notions
(various conformal algebras, their representations and cohomologies)
are higher-level analogues of the ordinary notions
in the pseudo-tensor category \cite{BD2004}
associated with
the polynomial Hopf algebra $H=\Bbbk[D]$, see \cite{BDK2001} for
details. Note that ordinary algebras (representations, cohomologies)
correspond to the case of 1-dimensional Hopf algebra.

For conformal algebras, the analogues of finite-dimensional
algebras are those finitely generated as modules over $H=\Bbbk[D]$.
These conformal algebras are called finite.
The structure theory of finite Lie conformal algebras was developed
in \cite{DK1998} and later generalized in \cite{BDK2001} for pseudo-algebras.
The structure
theorems for finite associative conformal algebras and pseudo-algebras
were derived from the corresponding statements on Lie algebras.

If an ordinary (Lie or associative) algebra $A$ has a finite-dimensional faithful
representation then $A$ is obviously finite-dimensional itself.
This is not the case for conformal algebras:
If a (Lie or associative) conformal algebra $C$ has finite
representation (i.e., a representation on a finitely generated $H$-module)
which is faithful then $C$ may not be finite itself. This phenomenon
brings into consideration another class of ``small''
 conformal algebras---those
with finite faithful representations.
The structure of such associative conformal algebras was described
in \cite{K.2006b}. For Lie conformal algebras with
finite faithful representations, the classification problem
was partially solved in \cite{BKL2003, DSK2002, Zelmanov2003}, but
 remains open in general.

However, it is not clear that the class of all conformal algebras
that have a finite faithful representation includes the class of finite
conformal algebras.

For associative conformal algebras, this question is closely
related with another one (posed in \cite{Retakh2006}): Is it possible
to join a conformal unit (introduced in \cite{Retakh2001})
to an associative conformal algebra?
The last problem is interesting itself since
unital conformal algebras have a good structure raising from
ordinary differential algebras (see Theorem \ref{thm:Retakh}).

In this paper, we solve both these problems for associative
conformal algebras. It turns out that every finite associative
conformal algebra can be embedded into an associative conformal
algebra with a (two-sided) conformal unit. In particular,
every finite associative conformal algebra has a finite faithful
representation. For infinite associative conformal algebras,
these statements are not true: We state examples of associative
conformal algebras that can not be embedded into unital ones
(even with a one-sided unit)
and have no finite faithful representation.

For finite Lie conformal algebras, the problem of
existence of a finite faithful representation is an analogue
of the classical Ado Theorem for finite-dimensional Lie algebras.
If $L$ is a centerless (e.g., semisimple) finite Lie conformal algebra
then the regular (adjoint) representation of $L$ on itself is finite and faithful.
However, this is unknown if every finite Lie conformal algebra has a finite
faithful representation.
This is also unknown whether a finite Lie conformal
algebra can be embedded into an associative conformal algebra
(this is not true in general \cite{Roitman2000}).

In \cite{Roitman2005}, it was shown that a
nilpotent Lie conformal algebra
can be embedded into a nilpotent associative one.
Taking into account that finitely generated nilpotent associative
conformal algebra is finite, we may conclude that a finite nilpotent
Lie conformal algebra
has a finite faithful representation.
In this paper, we obtain a more general result: Every
solvable Lie conformal algebra has a finite faithful representation
(in particular, it can be embedded into a unital associative
conformal algebra).
We also show that if the conformal analogue of the Poincar\'e---Birkhoff---Witt
Theorem in the sense of \cite{Roitman2000} holds for a finite Lie conformal algebra $L$
then $L$ has a finite faithful representation.

\section{Preliminaries}

\subsection{Conformal algebras}
Suppose $C$ is a left unital module over the polynomial algebra
$H=\Bbbk [D]$ ($\mathrm{char}\,\Bbbk=0$) equipped by a countable
family of $\Bbbk $-bilinear {\em $n$-products} $(\cdot \oo{n}\cdot )$, $n$
ranges over the set $\mathbb Z_+$ of non-negative integers. Then
$C$ is said to be a {\em conformal algebra\/} \cite{Kac1996} if
these operations satisfy the following axioms:
\begin{gather}
a\oo{n} b = 0 \ \mbox{for almost all } n\ge 0,
 \label{eq:C1} \\
Da\oo{n} b = -na\oo{n-1}b, \label{eq:C2} \\
a\oo{n} Db = D(a\oo{n} b) + na\oo{n-1} b \label{eq:C3}
\end{gather}
for all $a,b\in C$. (One should assume zero in the right-hand part
when the index of a product becomes negative.)

The axiom \eqref{eq:C1}, known as the {\em locality axiom}, allows to
define {\em locality function\/} $N_C: C\times C \to \mathbb Z_+$.
Namely, $N_C(a,b)$ is the minimal non-negative integer such that
$a\oo{n} b = 0$ for all $n\ge N_C(a,b)$. Axioms \eqref{eq:C2},
\eqref{eq:C3} are called {\em sesqui-linearity}.

A conformal algebra $C$ is said to be {\em torsion-free\/} if it
has no torsion as an $H$-module:
\[
\mathrm {Tor}_H C:= \{ a\in C \mid ha=0 \mbox{ for some } 0\ne
h\in H\} = 0.
\]
It was proved in \cite{Kac1997} that the torsion of a conformal
algebra $C$ is an annihilator ideal of $C$, i.e., $C\oo{n}
\mathrm{Tor}_H C = \mathrm{Tor}_H C\oo{n} C = 0$
 for all $n\ge 0$.

A very useful technique for computations in conformal algebras was
proposed in \cite{Kac1997}. Assume $\lambda $ is a formal
variable taking on its values in $\Bbbk $, and consider
\[
a\oo{\lambda } b = \sum\limits_{n=0}^{N_C(a,b)-1}
\dfrac{\lambda^n}{n!}(a\oo{n} b) \in C[\lambda ] ,\quad a,b\in C.
\]
The polynomial $a\oo{\lambda }b$ with coefficients in $C$ is
called $\lambda$-{\em product\/} of $a$~and~$b$. Then
\eqref{eq:C2} and \eqref{eq:C3} can be equivalently written as
\[
Da\oo{\lambda } b  = -\lambda (a\oo{\lambda } b), \quad
a\oo{\lambda } Db = (D+\lambda )(a\oo{\lambda } b).
\]

For $a,b\in C$, $n\ge 0$ denote
\begin{equation}\label{eq:Braced-n-prod}
\{a\oo{n} b\} =
 \sum\limits_{s\ge 0} \dfrac{(-1)^{n+s}}{s!} D^s(a\oo{n+s} b).
\end{equation}
The sum is finite because of \eqref{eq:C1}. In terms of
$\lambda$-product,
the elements $\{a\oo{n} b\}$ are the coefficients of
$$
\{a\oo{\lambda } b\} := \sum\limits_{n\ge 0}
 \dfrac{\lambda ^n}{n!} \{a\oo{n} b\} = (a\oo{-D-\lambda } b)\in
 C[\lambda ].
$$

A conformal algebra $C$ is said to be associative if
\begin{equation}\label{eq:ConfAs}
a\oo{\lambda } (b\oo{\mu } c) = (a\oo{\lambda } b)\oo{\lambda+\mu}
c
\end{equation}
or, in terms of $n$-products,
\begin{equation}\label{eq:ConfAs-n}
a\oo{n} (b\oo{m} c) =
 \sum\limits_{s\ge 0} \binom{n}{s} (a\oo{n-s} b)\oo{m+s} c
\end{equation}
for all $a,b,c\in C$, $n,m\ge 0$.
If
\begin{gather}
(a\oo{\lambda } b) = -\{b\oo{\lambda } a\}, \label{eq:ConfA-Comm} \\
a\oo{\lambda }(b\oo{\mu } c) - b\oo{\mu} (a\oo{\lambda } c)
 = (a\oo{\lambda } b)\oo{\lambda +\mu} c
                      \label{eq:ConfJacobi}
\end{gather}
for all $a,b,c\in C$
then $C$ is said to be a Lie conformal algebra. These properties
can be expressed in terms of $n$-products as follows:
\begin{gather}
(a\oo{n} b) = -\{b\oo{n} a\}, \label{eq:ConfA-Comm-n} \\
a\oo{n}(b\oo{m} c) - b\oo{m} (a\oo{n} c)
 = \sum\limits_{s\ge 0}
  \binom{n}{s}(a\oo{n-s} b)\oo{m+s} c
                      \label{eq:ConfJacobi-n}
\end{gather}
for all $n,m\ge 0$.

In \cite{BDK2001}, these notions were explained in terms of
pseudo-tensor categories.

One may consider $\{\cdot\oo{n} \cdot\}$, $n\ge 0$, as a new
family of operations on~$C$. It is easy to verify that
\[
\{Da\oo{\lambda } b\}  =(D+\lambda )\{a\oo{\lambda } b\}, \quad
\{a\oo{\lambda } Db\} = -\lambda \{a\oo{\lambda } b\}.
\]

If $C$ is an associative conformal algebra then these operations
have the following properties (see \cite{Kac1997}):
\begin{gather}
a \oo{\lambda } \{ b \oo{\mu} c\}  =
 \{(a \oo{\lambda } b) \oo{\mu} c\};   \label{eqCC1} \\
\{a \oo{\lambda } (b \oo{\mu} c)\}  =
   \{\{a \oo{\mu } b\} \oo{\lambda -\mu} c \}  ;   \label{eqCC2} \\
\{ a \oo{\lambda }\{b \oo{\mu} c \}\}   =
   \{\{ a\oo{\lambda -\mu} b\} \oo{\mu} c \};    \label{eqCC3} \\
\{a \oo{\lambda } b\} \oo{\mu } c =
  a\oo{\mu-\lambda }(b \oo{\lambda } c),
  \label{eqCC4}
\end{gather}
for $a,b,c\in C$.

\begin{exmp}
If $A$ is an ordinary (associative or Lie) algebra then
the free $H$-module $C=H\otimes A$ is an (associative or Lie)
conformal algebra with respect to
\[
(f(D)\otimes a)\oo{\lambda } (g(D)\otimes b)
 = f(-\lambda)g(D+\lambda )\otimes ab, \quad
 f,g\in H,\ a,b\in A.
\]
This structure is called the {\em current conformal algebra\/}
$\Curr A$ over~$A$.
\end{exmp}

\begin{exmp}\label{exmp:WeylConf}
Consider the free $H$-module
$C=H\otimes \Bbbk[x]\simeq \Bbbk[D,x]$. Then the $\lambda $-product
\[
f(D,x)\oo{\lambda } g(D,x) = f(-\lambda , x)g(D+\lambda , x+\lambda ),
\quad f,g\in \Bbbk[D,x],
\]
turns $C$ into an associative conformal algebra known as the
{\em Weyl conformal algebra}.
\end{exmp}

\begin{exmp}
Consider the free 1-generated $H$-module
generated by an element $x$. This module,
$H\otimes \Bbbk x$, with respect to the $\lambda $-product
\[
(f(D)\otimes x)\oo{\lambda } (g(D)\otimes x)
=f(-\lambda )g(D+\lambda )(D+2\lambda )\otimes x
\]
is a Lie conformal algebra called the {\em Virasoro conformal algebra}
$\mathcal Vir$.
\end{exmp}

Given an associative conformal algebra $C$, one may consider the
same $H$-module endowed with new operations $(\cdot
\oo{n}^{\mathrm{op}}\cdot )$, $n\ge 0$, defined as follows:
\[
a\oo{n}^{\mathrm{op}} b = \{b\oo{n} a\} , \quad a,b\in C.
\]
Then the system $C^{\mathrm{op}}=(C, D, (\cdot \oo{n}^{\mathrm{op}}\cdot), n\ge 0 )$ is
an associative conformal algebra called opposite
to~$C$. It is clear that $(C^{\mathrm{op}})^{\mathrm{op}} = C$.

An element $e$ of an associative conformal algebra $C$ is said to
be a (left) {\em unit} \cite{Retakh2001} if $e\oo{0} x=x$ for all
$x\in C$ and $N_C(e,e)=1$. It is natural to introduce the
``opposite'' notion: $e$ is said to be a {\em right unit\/} of $C$ if
$\{x\oo{0} e\}=x$ for all $x\in C$ and $N_C(e,e)=1$. It is clear
that if $e$ is a left (right) unit of $C$ then $e$ is a right
(left) unit of~$C^{\mathrm{op}}$. If a left unit $e\in C$ is a right unit
of $C$ then it is said to be a two-sided unit of~$C$.

Relations \eqref{eqCC2} and \eqref{eqCC4} imply that if
$e$ is a left (right) unit of an associative conformal algebra $C$
then $(a-\{a\oo{0} e \})\oo{\lambda} C=0$
(respectively, $\{C\oo{\lambda } (a-e\oo{0} a)\}= 0$).
Therefore, if the left (right) annihilator of $C$ is zero
then each left (right) unit is two-sided.
For example, this is the case for
semisimple conformal algebras.
The following statement was originally proved in a
slightly weaker form.

\begin{thm}[\cite{Retakh2001}]\label{thm:Retakh}
Let $C$ be an associative conformal algebra with a two-sided unit.
Then there exists an ordinary associative algebra $A$ with a locally
nilpotent derivation $\partial $ such that $C\simeq H\otimes A$ with
respect to the $\lambda $-product is given by
\[
(f(D)\otimes a)\oo{\lambda } (g(D)\otimes b)
 = f(-\lambda )g(D+\lambda )\otimes a e^{\lambda \partial }(b),
 \quad a,b\in A,\ f,g\in H.
\]
\end{thm}

In particular, $a\oo{n} b = a\partial^n(b)$ for $a,b\in A$.
Conformal algebras of this type are called differential \cite{Retakh2001}.

If $C$ is an associative conformal algebra then the same
$H$-module with respect to the new products
\begin{equation}\label{eq:ConfCommutator}
[a\oo{n} b]=a\oo{n} b - \{b\oo{n} a\}, \quad a,b\in C,\ n\ge 0,
\end{equation}
is a Lie conformal algebra denoted by $C^{(-)}$. It is easy to
deduce from \eqref{eq:ConfAs}, \eqref{eqCC1}--\eqref{eqCC4} that
\begin{equation}\label{eq:CommDer}
[a\oo{\lambda } (b\oo{\mu} c)]
 = [a\oo{\lambda }b]\oo{\lambda +\mu} c
 +b\oo{\mu } [a\oo{\lambda } c].
\end{equation}

Given a Lie conformal algebra $L$, one may consider the category
$\mathcal E(L)$ of associative envelopes of $L$. The objects of
$\mathcal E(L)$ are pairs $(C,\iota)$, where $C$ is an associative
conformal algebra, $\iota : L\to C^{(-)}$ is a homomorphism of
conformal algebras, and $C$ is generated (as an associative
conformal algebra) by the set $\iota (L)$. Suppose $(C_1,\iota_1)$
and $(C_2,\iota_2)$ are two objects of $\mathcal E(L)$. A
homomorphism $\varphi: C_1\to C_2$ of conformal algebras is a
morphism of
$\mathcal E(L)$ if $\varphi\iota_1 = \iota_2: L\to C_2^{(-)}$.

The natural analogues of universal associative enveloping
algebras for Lie conformal algebras were introduced in \cite{Roitman2000}.
Suppose $L$ is a Lie conformal algebra generated by its subset $B$,
and let $N$ be a function from $B\times B$ to $\mathbb Z_+$.
Then there exists an associative envelope $(U_N(L), \iota_N)\in \mathcal E(L)$
satisfying the following property: For every
$(U,\iota)\in \mathcal E(L)$ such that $N_U(\iota(a),\iota(b))\le N$ for all
$a,b\in B$ there exists a morphism of envelopes
$\varphi : U_N(L)\to U$.

It was shown in \cite{Roitman2000} that there exist Lie
conformal algebras $L$ such that $\iota$ is not injective for every
$(C,\iota)\in \mathcal E(L)$, i.e., $L$ can not be embedded into
an associative conformal algebra.

The classical Poincar\'e---Birkhoff---Witt (PBW) Theorem for Lie algebras
states that the graded associative algebra $\mathrm{gr}\,U(\mathfrak g)$
(with respect to the natural
filtration) of the universal enveloping
algebra $U(\mathfrak g)$ of a Lie algebra $\mathfrak g$
is isomorphic to the symmetric algebra $S(\mathfrak g)$.
For conformal algebras, the analogue of the PBW Theorem would sound as follows
\cite{Roitman2000}.

If $L$ is a torsion-free finite Lie conformal algebra with a basis
$B$ over $H$ then every associative envelope $(U,\iota)\in \mathcal E(L)$
has a natural filtration
$U_1\subseteq U_2\subseteq \dots $
defined as follows:
$U_n$ is spanned over $H$ by all terms
$\iota(a_1)\oo{n_1}\dots \oo{n_{k-1}}\iota(a_k)$, $a_i\in B$, $1\le k\le n$,
$n_i\ge 0$, with any bracketing.
Since
$U_n\oo{\lambda } U_m\subseteq U_{n+m}[\lambda ]$ for all $n,m\ge 1$,
the graded associative conformal algebra
$\mathrm{gr}\,U$ can be defined in the ordinary way.

In \cite{Roitman2000}, the notion of the
{\em PBW property\/} for torsion-free
conformal Lie
algebras (with respect to a fixed basis $B$ over $H$) was
proposed: An algebra $L$ has the PBW property with respect
to $B$ if for any sufficiently large constant locality function
$N$ on $B\times B$
the graded universal enveloping associative
conformal algebra $\mathrm{gr}\, U_N(L)$ is isomorphic to
the free commutative conformal algebra generated by $B$ with
respect to the locality function~$N$.
Unfortunately, not all finite torsion-free Lie conformal algebras
have the PBW property.

\subsection{Conformal endomorphisms and representations of conformal algebras}
Let us recall the notion of a conformal endomorphism
\cite{Kac1997}. Consider a left unital module $M$ over $H$. By
$\End M$ we denote the associative algebra of $\Bbbk$-linear maps from
$M$ to itself.
A left (right) {\em conformal endomorphism\/} of $M$ is a
sequence $a=\{a_n\}_{n\ge 0}$, $a_n\in \End M$, such that for
every $u\in M$
we have $a_n(u) = 0$ for almost all $n\ge 0$, and
$[a_n, D]= na_{n-1}$ (respectively, $a_nD = -n a_{n-1}$).

Denote by $\Cend^l M$ ($\Cend^r M$) the set of all left (right)
conformal endomorphisms of $M$. These are linear spaces which are
also $H$-modules with respect to the actions defined by
\[
\begin{gathered}
(Da)_n = -na_{n-1},\quad a\in \Cend^l M,\ n\ge 0; \\
(Da)_n = Da_n + na_{n-1}, \quad a\in \Cend^r M, \ n\ge 0.
\end{gathered}
\]

If $M$ is a finitely generated $H$-module then both $\Cend^l M$
and $\Cend ^r M$ are associative conformal
algebras with respect to the following operations \cite{Kac1997}:
\[
\begin{gathered}
(a\oo{n} b)_m = \sum\limits_{s=0}^n (-1)^s\binom{n}{s}
a_{n-s}b_{m+s},
 \quad a,b\in \Cend^l M,\ n,m\ge 0; \\
(a\oo{n} b)_m= \sum\limits_{s=0}^m \binom{m}{s} b_{n+s} a_{m-s} ,
\quad
 a,b\in \Cend^r M,\ n,m\ge 0;
\end{gathered}
\]
These operations can be defined on $\Cend^l M$ and $\Cend^rM$ for
an arbitrary $H$-module $M$. Finiteness condition is needed to
guarantee the locality axiom \eqref{eq:C1} to hold.
If $M$ is an infinitely generated $H$-module then \eqref{eq:C1}
may not be true in $\Cend^l M$ or $\Cend^r M$, but it is still possible to define
$\{\cdot\oo{n}\cdot\}$, $n\ge 0$, via \eqref{eq:Braced-n-prod}: Infinite sum
converges to a conformal endomorphism. Therefore, the operations
\eqref{eq:ConfCommutator} are well-defined on $\Cend^l M$
or $\Cend^r M$.
 If $M$ is
finitely generated over $H$ then $(\Cend^l M)^{(-)}$
is a Lie
conformal algebra denoted by $\mathrm{gc}\, M$.

One may consider a conformal endomorphism $a\in \Cend^l M$ as a
formal power series
$$
a_\lambda = \sum\limits_{n\ge 0} \dfrac{\lambda ^n}{n!} a_n \in
(\End M)[[\lambda ]],
$$
where $a_\lambda (u) \in M[\lambda ]$ for every $u\in M$ and
$a_\lambda D = (D+\lambda)a_\lambda $. Then the operations of a
conformal algebra can be defined in terms of formal series as
\[
(Da)_\lambda = -\lambda a_\lambda, \quad (a\oo{\lambda } b)_\mu =
a_\lambda b_{\mu -\lambda},\quad a,b\in \Cend^l M.
\]

For $\Cend^r M$, a similar description can be stated. In this
case, a formal series $a_\lambda \in (\End M)[[\lambda ]]$
defines a right conformal endomorphism if
$a_\lambda (u) \in M[\lambda ]$ for every $u\in M$ and
 $a_\lambda D = -\lambda a_\lambda$. Conformal
operations are defined by
\[
(Da)_\lambda = (D+\lambda )a_\lambda , \quad (a\oo{\lambda }
b)_\mu =  b_{\lambda+\mu }a_{\mu},
\quad a,b\in \Cend^r M.
\]

\begin{prop}\label{prop:CendOp}
There exists an $H$-linear bijection
$\tau: \Cend^l M\to \Cend^r M$
such that
$\{\tau(b)\oo{n}\tau(a) \} = \tau(a\oo{n} b)$
for all $a,b\in \Cend^l M$, $n\ge 0$.
In particular, for a finitely generated $H$-module $M$
associative conformal algebras
$\Cend^l M$ and $(\Cend^r M)^{\mathrm{op}}$
are isomorphic.
\end{prop}

\begin{proof}
Consider the map
\[
\begin{aligned}
&\tau: \Cend^l M \to \Cend^r M, \\
&\tau(a)_n = \sum\limits_{s\ge 0} \frac{(-1)^{n+s}}{s!} D^sa_{n+s},\quad
\quad a\in \Cend^l M,\ n\ge 0.
\end{aligned}
\]
The sum in this definition is formally infinite, but $\tau(a)_n$
is a well-defined linear map from $M$ to $M$ since for every $u\in M$ only a
finite number of $a_n(u) $ are nonzero.

In terms of formal power series, $\tau $ may be expressed as
$$
\tau(a)_\lambda = a_{-D-\lambda }.
$$

Let us check that $\tau(a)\in \Cend^r M$. Indeed,
$$
\tau(a)_\lambda D = a_{-\lambda - D} D =
(D+(-\lambda-D))a_{-\lambda-D} = -\lambda \tau(a)_\lambda .
$$
Moreover,
$\tau(Da)_\lambda = (Da)_{-D-\lambda}  = (D+\lambda )a_{-D-\lambda }
=(D+\lambda )\tau(a)_\lambda =(D\tau(a))_\lambda$,
so $\tau $ is $H$-linear.
Note that
$\tau $ is a bijection since $\tau^{-1}$ can be defined
by the very same formula: $\tau^{-1}(b)_\lambda =b_{-\lambda -D}$,
$b\in \Cend^r M$.

 Finally,
\begin{multline}\nonumber
\{\tau(b)\oo{\lambda } \tau(a)\}_\mu = (\tau(b)\oo{-D-\lambda }
\tau(a))_\mu = \sum\limits_{n \ge 0}
  \dfrac{1}{n!}((-D-\lambda)^{n} (\tau(b)\oo{n} \tau(a) ))_\mu  \\
= \sum\limits_{n \ge 0}
  \dfrac{1}{n!}(-D-\mu -\lambda)^{n} (\tau(b)\oo{n} \tau(a) )_\mu
= \sum\limits_{n,m,s\ge 0}
 \dfrac{1}{n!}\dfrac{1}{m!}
  (-D-\mu-\lambda)^n \mu^m \binom{m}{s} \tau(a)_{n+s} \tau(b)_{m-s} \\
= \sum\limits_{n,m,s\ge 0}
 \dfrac{1}{n!}\dfrac{1}{(m-s)!}\dfrac{1}{s!}
  (-D-\mu-\lambda)^n \mu^s \mu^{m-s}\tau(a)_{n+s} \tau(b)_{m-s} \\
= \sum\limits_{n,s\ge 0}
 \dfrac{1}{n!}\dfrac{1}{s!}(-D-\mu-\lambda)^n \mu^s \tau(a)_{n+s}
 \sum\limits_{m\ge 0}\dfrac{1}{m!}\mu^m \tau(b)_{m}
= \tau(a)_{-D-\lambda } \tau(b)_\mu =a_{\lambda }b_{-D-\mu }
\end{multline}
for every $a,b\in \Cend^l M$. On the other hand,
\begin{multline}\nonumber
\tau(a\oo{\lambda } b)_\mu = (a\oo{\lambda } b)_{-D-\mu } =
\sum\limits_{n,m\ge 0}
 \dfrac{1}{n!}\dfrac{1}{m!} (-D-\mu)^m \lambda ^n (a\oo{n} b)_m \\
= \sum\limits_{n,m,s\ge 0}
  \dfrac{1}{n!}\dfrac{1}{m!} (-1)^s
   \binom{n}{s} (-D-\mu)^m \lambda ^n a_{n-s}b_{m+s} \\
= \sum\limits_{m,s\ge 0}
 \dfrac{1}{m!}\dfrac{1}{s!} (-\lambda )^s(-D-\mu)^m
\left( \sum\limits_{n\ge 0} \dfrac{1}{n!}\lambda ^na_n\right)
  b_{m+s}
= \sum\limits_{m\ge 0}\dfrac{1}{m!} (-D-\mu-\lambda )^m  a_\lambda
b_m
 \\
= \sum\limits_{m\ge 0}\dfrac{1}{m!}
  a_\lambda (-D-\mu)^m b_m
= a_\lambda b_{-D-\mu }.
\end{multline}
\end{proof}

\begin{defn}[\cite{CK1997}]
A left (right) {\em representation\/}
of an associative conformal
algebra $C$ on an $H$-module $M$ is an $H$-linear map $\rho $ from $C$ to
$\Cend^l M$ ($\Cend^r M$) such that $\rho(a\oo{n} b) =
\rho(a)\oo{n} \rho(b)$ for all $a,b\in C$, $n\ge 0$.

For a Lie conformal algebra $C$, its representation is an
$H$-linear map  $\rho : C\to \Cend^l M$ such that
$\rho (a\oo{n} b) = [\rho(a)\oo{n} \rho(b)]$,
$a,b\in C$, $n\ge 0$.
\end{defn}

Proposition \ref{prop:CendOp} implies that left (right)
representations of an associative conformal algebra $C$
are in one-to-one correspondence with right
(left) representations of $C^{\mathrm{op}}$.
 Obviously,
there is no reason to distinguish left and right representations
for Lie conformal algebras.

A left representation $\rho $ of a (Lie or associative) conformal
algebra $C$ on an $H$-module $M$ can be described in terms of
operations $(\cdot \oo{n}\cdot ): C\otimes M\to M$, $n\ge 0$, defined as follows:
\[
a\oo{n} b  = \rho(a)_n b,\quad a\in C,\ b\in M.
\]
These operations
satisfy the analogues of \eqref{eq:C1}--\eqref{eq:C3} and
\eqref{eq:ConfAs-n} (for associative algebra) or
\eqref{eq:ConfJacobi-n} (for Lie algebra).
We will also use the formalism of $\lambda $-products for
representations of conformal algebras.

If $C$ has a left (right) representation on an $H$-module $M$ then
$M$ is called left (right) {\em conformal $C$-module}.

A representation $\rho $ (or the corresponding conformal module $M$)
is said to be {\em finite\/} if $M$ is a finitely generated $H$-module.
If $\rho $ is injective then $M$ is called faithful conformal module.

Consider a free finitely generated $H$-module $M$. If we fix a
system of $H$-linearly independent generators $e_1,\dots, e_n$ of
$M$ then $M$ can be identified with the linear space $H\otimes
\Bbbk^n$, and $\Cend^l M$ can be presented as the space of
matrices $\mathbb M_n(\Bbbk [D,x])\simeq \Bbbk[D,x]\otimes \mathbb
M_n(\Bbbk)$, where
\begin{equation}\label{eq:Cend-Action}
(f(D,x)\otimes A)_\lambda (h(D)\otimes u) = f(-\lambda, D)h(D+\lambda )\otimes  Au
\end{equation}
for $f\in \Bbbk[D,x]$, $A\in \mathbb M_n(\Bbbk )$, $ h\in H$, $u\in \Bbbk^n$. Then
\begin{equation}\label{eq:Cend-Product}
(f(D,x)\otimes A)\oo{\lambda } (g(D,x)\otimes B) =
f(-\lambda ,x)g(D+\lambda ,x+\lambda )\otimes AB
\end{equation}
for $f,g\in \Bbbk[D,x]$,
$A,B\in \mathbb M_n(\Bbbk )$ \cite{Kac1997, Retakh2001}.
It means that $\Cend^l M$ is a differential conformal algebra
(see Theorem \ref{thm:Retakh})
based on the algebra $A=\mathbb M_n(\Bbbk [x])$ with respect to
the ordinary derivation $\partial = d/dx$.

It is easy to note (see, e.g., \cite{K.2006a} that $\Cend^l_n
\simeq \Cend_n^r$. From now on, we will denote by $\Cend_n$ the
conformal algebra of left conformal endomorphisms of $H\otimes
\Bbbk^n$. In particular, $\Cend_1$ is just the Weyl conformal algebra
from Example \ref{exmp:WeylConf}.

Left and right ideals of $\Cend_n$ were described in
\cite{BKL2003}. All right ideals are of the form $\Cend_{P,n} = P(x)\mathbb
M_n(\Bbbk[D,x])$, where $P(x)$ is a matrix in $\mathbb
M_n(\Bbbk[x])$; all left ideals are of the form $\Cend_{n,P} =
\mathbb M_n(\Bbbk[D,x]) P(x-D)$.

\subsection{Free associative conformal algebras}
The class of associative conformal algebras is not a variety in
the ordinary sense. However, given a set $B$ of generators and a
function $N:B\times B \to \mathbb Z_+$, one may build free
associative conformal algebra $\ConfAs\langle B; N\rangle $
generated by $B$ with respect to the locality function $N$
\cite{Roitman1999}. This algebra has the following universal
property: For every associative conformal algebra $C$ and for
every map $\alpha : B\to C$ such that $N_C(\alpha(a),
\alpha(b))\le N(a,b)$ for all $a,b\in B$ there exists a unique
homomorphism $\varphi : \ConfAs\langle B; N\rangle \to C$ such
that $\varphi\vert_B =\alpha $.

The structure of $\ConfAs\langle B;N\rangle $ was completely
described in \cite{Roitman1999}. In particular, its linear basis
consists of {\em conformal monomials}
$D^s (a_1\oo{n_1} (a_2\oo{n_2} \dots (a_k\oo{n_k} a_{k+1})\dots ))$,
where $k,s\ge 0$, $a_i\in B$, $0\le n_i< N(a_i,a_{i+1})$.
It is easy to see that the right-justified monomials of the form
\begin{equation}\label{eq:RightNormedBasis}
u=D^s ((\dots (a_1\oo{n_1} a_2) \dots a_k)\oo{n_{k}}a_{k+1})
\end{equation}
with the same restrictions on and $n_i$
also form a linear basis of $\ConfAs\langle B;N\rangle $.
It is natural to denote
$k=\deg_B u$, $s=\deg_D u$. For every {\em conformal polynomial}
\[
f=\sum\limits_i \alpha_i u_i\in \ConfAs\langle B;N\rangle , \quad
\alpha_i\in \Bbbk,
\]
where each $u_i$ is of the form \eqref{eq:RightNormedBasis},
$\deg_B f$ ($\deg_D f$) is the maximal $\deg_B u_i$ ($\deg_D u_i$)
among all $u_i$ such that $\alpha _i\ne 0$.
If $\deg_D f = 0$ then $f$ is said to be {\em $D$-free\/}
conformal polynomial.

Assume $B$ is a linearly ordered set. Consider the following order
on the basic conformal monomials: If $u$ is of the form
\eqref{eq:RightNormedBasis} then define its {\em weight} as
$$
\mathrm {wt}\,(u) = (k+1,a_{k+1}, n_{k},a_{k}, \dots ,
a_2, n_1, a_1, s)
$$
and compare weights lexicographically, i.e.,
\begin{equation}\label{eq:RightNormedOrder}
u_1\le u_2 \iff \mathrm{wt}\,(u_1) \le \mathrm{wt}\,(u_2).
\end{equation}

\subsection{Gelfand---Kirillov dimension of conformal algebras}
In group and ring theory, the notions of growth of a system
(group, algebra, or module) and Gelfand---Kirillov dimension (as a
logarithmic measure of growth) play an important role. For
conformal algebras, the analogous notions were introduced in
\cite{Retakh2001}.

Suppose $C$ is a finitely generated conformal algebra, $V$ is a fixed
set of generators. Then, for every integer $n\ge 1$, the
$H$-submodule $V(n)$ of $C$ spanned by all conformal monomials in
$V$ of length $l\le n$ is finitely generated over $H$.
The {\em Gelfand---Kirillov dimension} (GK-dimension) of $C$ is defined as
$$
\mathrm{GKdim }\, C = \overline{\lim\limits_{n\to \infty}}
 \dfrac{\ln \mathrm{rank}\,V(n)}{\ln n}.
$$
This characteristic (real number or $\infty$) does not depend on
the choice of a generating set~$V$. If $C$ is infinitely
generated then its GK-dimension is equal to the
supremum of GK-dimensions of its finitely
generated subalgebras.

In particular, if $\mathrm{GKdim }\, C=0$ then $C$ is locally
finite. There are no conformal algebras of intermediate GK-dimension
between 0 and 1. If $\mathrm{GKdim}\, C=1$ then $C$ is said to be
of {\em linear growth}. For example, $\mathrm{GKdim}\,\Cend_n=1$.
Therefore, if an (associative or Lie) conformal algebra has a finite
faithful representation then its GK-dimension is equal to 0 or 1.

It was conjectured in \cite{Zelmanov2003} and proved
in \cite{K.2006a, K.2006b} that
the class of simple finitely generated associative conformal algebras of linear
growth coincides with the class of infinite simple associative conformal algebras
with finite faithful representation:
All these algebras are of the form $\Cend_{n,P}$, $\det P\ne 0$.
It is natural to conjecture that
all finitely generated associative conformal algebras of linear growth have
a finite faithful representation. As we show in this note, this is not true.

\subsection{Summary of results}

Consider the following statements about a torsion-free
associative conformal algebra $C$.

\begin{enumerate}

\item[(F)] $C$ is finite;
\item[(FR)] $C$ has a finite faithful left representation;
\item[(FFR)] $C$ has a finite faithful
left representation on a free $H$-module;
\item[(U)] $C$ can be embedded into an associative conformal algebra
that contains a right unit;
\item[(TU)] $C$ can be embedded into an associative conformal algebra
that contains a two-sided unit;
\item[(NA)] $C$ has no right annihilator,
i.e., $\{a\in C\mid C\oo{\lambda } a = 0 \}=0$.
\end{enumerate}

It is trivial that (TU) implies (U) and
(FFR) implies~(FR).
Also, it follows from the structure of $\Cend_n$
that (FFR) implies~(TU): The element $1\otimes I_n$, where $I_n$ is the
unit matrix, is a two-sided unit of $\Cend_n$.
It is easy to see (e.g., \cite{K.2006b}) that
(FR) together with~(NA) imply~(FFR).
In the same paper, it was shown that free associative
conformal algebra with uniform locality on generators
satisfies~(TU).

Moreover, it is easy to see that (F) and (U) imply~(FFR) and~(TU).
Indeed, if $C\subseteq C'$, where $C'$ contains a right unit $e$,
then $C'$ is torsion-free, and the left $C$-submodule $M$
generated in
$C'$ by the element $e$ is finite over $H$.
The corresponding left representation of $C$ on $M$ is
faithful since $e\in M$.

In this note, we will show
that (U) does not imply~(TU),
and (FR) does not imply~(FFR), in general.

We state an example of a finitely generated associative
conformal algebra $C$ of linear
growth that satisfies~(NA), but does not satisfy~(U).
In particular, $C$ does not satisfy~(FR).

Another example shows that even a locally finite torsion-free
associative conformal algebra $C$ may have no faithful finite
representation: We construct an example of such $C$ that does
not satisfy neither~(U), nor~(FR).

\begin{rem}
It is clear that all these statements above remain valid if
we exchange the attributes ``left'' and ``right'' for
representations, units, and annihilators.
\end{rem}

Our main result for associative conformal algebras
states that (F) implies (FFR). In particular, every
torsion-free finite associative conformal algebra
satisfies~(TU).
Being combined with the result of \cite{Roitman2005}, the
statement obtained allows to conclude that finite
nilpotent Lie conformal algebra has a finite faithful representation
on a free $H$-module.
We use another method to obtain more general fact: A finite solvable
Lie conformal algebra has a finite faithful representation
on a free $H$-module.

Also, we point out the following curious fact. For a finite Lie conformal algebra,
the Poincar\'e---Birkhoff---Witt Theorem (in the sense of
\cite{Roitman2000}) implies the existence of a finite faithful representation.

\section{Finite representations of associative conformal algebras}

In this section, we prove that a finite associative conformal
algebra $C$
has a finite faithful conformal module $M$ which is free
as an $H$-module. In particular, $C$ can be embedded into
$\Cend _n$ for an appropriate~$n$.

Next, we state some examples of torsion-free associative conformal
algebras that do not satisfy neither (U), nor (FR). One example is
a finitely generated algebra of linear growth, another one is
infinitely generated but locally finite (of zero GK-dimension).

Throughout the rest of the paper, the term ``representation'' stands for
left conformal representation.

\subsection{Finite conformal algebras}

\begin{thm}\label{thm:Finite_Ass}
Let $C$ be a torsion-free finite associative conformal algebra.
Then $C$ has a faithful representation on a
free finitely generated $H$-module.
\end{thm}

\begin{proof}
Let $B$ be a finite basis of $C$ over $H$. Then
\[
C\simeq \ConfAs\langle B;N\rangle /I,
\]
where $N: B\times B\to \mathbb Z_+$ is an appropriate locality
function,
$I$ is an ideal of the free associative conformal algebra
$\ConfAs\langle B;N\rangle$.

\begin{lem}\label{lem:IProperties}
{\rm (i)}
If $0\ne f\in I$ then $\deg_B f>1$.

{\rm (ii)}
For every $g\in \ConfAs\langle B;N\rangle \setminus I$
there exists unique $\hat g\in \ConfAs\langle B;N\rangle$
such that $\deg_B \hat g =1$ and $g-\hat g\in I$.

{\rm (iii)} There exists a constant $M$ such that
$\deg_D \hat u\le M$ for all $D$-free
$u\in \ConfAs\langle B;N\rangle $.
\end{lem}

\begin{proof}
(i) The images of elements of $B$ in $C$
are linearly independent over~$H$. Therefore, if
$f=\sum_i h_i(D)a_i \in I$ for some $h_i\in H$, $a_i\in B$ then
$f=0$.

(ii) It follows from the fact that $C$ is generated by $B$
as an $H$-module. The statement (i) implies uniqueness.

(iii)
Consider the multiplication table of $C$: All $n$-products of
elements $a,b\in B$ are equal in $C$ to $H$-linear
combinations of elements of~$B$.
Therefore,
\[
I\ni a \oo{n} b - \sum\limits_{c\in B} h_{a,b}^{c,n} (D) c,
 \quad a,b\in B,\
0\le n < N(a, b),\ h_{a,b}^{c,n} \in H.
\]
Denote by $M$ the maximal $\deg h_{a,b}^{c,n}$
that appears in this table.

It is enough to show that $\deg_D \hat u\le M$ for
all $D$-free monomials $u$ of the form
\eqref{eq:RightNormedBasis}.
Assume
\[
\mathrm{wt}\,(u) = (k,a_k, n_{k-1}, a_{k-1},\dots, n_1, a_1, 0).
\]
Proceed by induction on $k\ge 1$.
For $k=1,2$ the statement is obvious.
For $k>2$,
$u = v\oo{n_{k-1}} a_{k}$,
where $v$ is a $D$-free monomial of length $k-1$.
Then
$u - \hat v\oo{n_{k-1}} a_{k}\in I$,
but
\eqref{eq:C2} implies $\hat v \oo{n_{k-1}} a_{k}$
to be a $D$-free conformal
polynomial of degree two in $B$ (or $\hat v\oo{n_{k-1}} a_k=0$).
Therefore,
$\deg_D \hat u\le M$.
\end{proof}

Denote by $B'=B\cup \{v\}$, where $v\notin B$.
Fix a non-negative integer constant $M'$ and define
a function $N': B'\times B' \to \mathbb Z_+$ in the following way:
\[
N'\vert _{B\times B} = N, \quad
N'(v,x)=0\ \mbox{for}\ x\in B',\quad
N'(b,v) = M'\ \mbox{for}\ b\in B.
\]

Note that $\ConfAs\langle B;N\rangle $
is a conformal subalgebra of
$ \ConfAs\langle B';N'\rangle$
for any $M'$.

\begin{lem}\label{lem:ZeroDiv}
Suppose $f\in \ConfAs\langle B;N\rangle $, $m\ge 0$.
If
$f\oo{m} v = 0$ in $\ConfAs\langle B';N'\rangle$
then
\[
f = f_1 + D^{m+1}f_2,\quad f_1,f_2\in  \ConfAs\langle B;N\rangle,
\ \deg_D f_1\le m-M'.
\]
\end{lem}

\begin{proof}
As an arbitrary element of  $\ConfAs\langle B;N\rangle$,
$f$ can be presented in the form
\[
f = \sum\limits_{s\ge 0} D^su_s,\quad \deg_D u_s = 0.
\]
Without loss of generality, we may assume $f$ to be homogeneous
of degree $k\ge 1$ in $B$.

Assume there exist $s\ge 0$ such that
$\max\{-1,m-M'\}<s\le m$ and $u_s\ne 0$.
Choose the minimal such~$s_0$.
Let
$$
(\dots ((a_1\oo{n_1} a_2)\oo{n_2} a_2)\oo{n_3} \dots )\oo{n_{k-1}} a_k
$$
be the principal term of $u_{s_0}$ with respect to the order
\eqref{eq:RightNormedOrder}.
Then
$$
((\dots ((a_1\oo{n_1} a_2)\oo{n_2} a_2)\oo{n_3} \dots )\oo{n_{k-1}} a_k)\oo{m-s_0} v
$$
is a basic monomial of
$\ConfAs\langle B';N'\rangle $ which is a principal term of
$f\oo{m} v$.
But $f\oo{m} v = 0$, i.e., it has no principal term.
\end{proof}

Denote by $J$ the ideal of
$\ConfAs\langle B';N'\rangle$ generated by~$I$, and let
$C_{M'}$ stands for the quotient conformal algebra
$\ConfAs\langle B';N'\rangle/J$.

\begin{lem}\label{lem:Embedding}
The initial conformal algebra $C$ is a subalgebra
of $C_{M'}$ for every $M'\ge 0$.
\end{lem}

\begin{proof}
Due to the definition of $N'$, the set
\[
I+H(I\cdot v)
:= \bigg\{f + \sum\limits_{s\ge 0} D^s(g_s\oo{n_s} v)\mid
 f, g_s\in I, \, n_s\ge 0 \bigg\}
\]
is an ideal of $\ConfAs\langle B';N'\rangle$. Obviously,
$I+H(I\cdot v)=J$.
But for every $s\ge 0$ either all terms of $D^s(g_s\oo{n_s} v)$
contain $v$ or $g_s\oo{n_s} v =0$. Therefore,
$J\cap \ConfAs\langle B;N\rangle =I$, that proves the lemma.
\end{proof}

\begin{lem}\label{lem:TorFree}
Conformal algebra $C_{M'}$ is a torsion-free $H$-module.
\end{lem}

\begin{proof}
Assume there exists $f\in \ConfAs\langle B'; N'\rangle $
such that $f\notin J$ but $h(D)f\in J$ for some nonzero
$h\in H$.

As an arbitrary element of $\ConfAs\langle B';N'\rangle $,
$f$ can be presented in a form
$$
f= f_0 + \sum\limits_{s=0}^k D^s(g_s\oo{n_s} v),\quad
f_0, g_s\in \ConfAs\langle B;N\rangle ,\ n_s\ge 0.
$$
Let us choose such an element $f$ with minimal $k$. Then $k>0$
(otherwise, $C$ contains nonzero torsion),
and $g_k\oo{n_k} v\notin J$.

Suppose $\deg h=m $. Then
$$
h(D) f = h(D)f_0 +D^{k+m}(g_k\oo{n_k} v) + \sum\limits_{s=0}^{k+m-1} D^su_s,
$$
where $u_s$ are $D$-free elements of $\ConfAs\langle B';N'\rangle $.
It was shown in the proof of Lemma \ref{lem:Embedding} that
$J=I+H(I\cdot v)$. Hence, in particular,
there exist $g\in I$ and $l\ge 0$ such that
$g_k\oo{n_k} v = g\oo{l}v\in J$, a contradiction.
\end{proof}

Hence, one may consider $C_{M'}$ as a regular left conformal
module over its subalgebra~$C$.  Denote by $U$ its
conformal $C$-submodule generated by~$v$.
This is a finitely generated torsion-free $H$-module, hence,
$C$ has a finite representation on the free $H$-module~$U$.

It remains to show that if $M'>M$, where $M$ is the constant
from Lemma~ \ref{lem:IProperties}(iii), then $U$ is a faithful
conformal $C$-module.

Assume there exists $f\in \ConfAs\langle B;N\rangle $
such that $f\notin I$ but $f\oo{n} v \in J$ for all $n\ge 0$.
By Lemma \ref{lem:IProperties}(ii), there exists
\[
\hat f = \sum\limits_{s\ge s_0} D^s u_s , \quad f-\hat f\in I,\
u_s\in \Bbbk B,
\]
where $u_{s_0}\ne 0$, $s_0\ge 0$. Then
$$
\hat f\oo{s_0} v = (-1)^{s_0}s_0! u_{s_0} \oo{0} v \in J.
$$

Hence, there exists $0\ne a\in \Bbbk B$ such that
$a\oo{0} v\in J$.
But $J=I+H(I\cdot v)$, so there exist $g\in I$ and $m\ge 0$
such that $ g\oo{m} v = a\oo{0} v$, or
\[
((-1)^m m! g - D^m a)\oo{m} v = 0.
\]

By Lemma \ref{lem:ZeroDiv},
\[
(-1)^m m! g - D^m a = \sum\limits_{s\ge 0} D^s u_s,
  \quad \deg_D u_s =0,
\]
where $u_s\ne 0$ only for $s>m$ or for $s\le m-M'$.

Consider
\[
\tilde g = \sum\limits_{s\ge 0} D^s \hat u_s.
\]
By definition,
\[
\tilde g - ((-1)^m m! g - D^m a) \in I,
\]
and by Lemma~\ref{lem:IProperties}
we have
$\deg_B \hat u_s =1$, $\deg_D \hat u_s\le M$
for all nonzero $\hat u_s$.
Therefore, all terms of $\tilde g$ are linear in $B$,
and their degrees in $D$
are either $>m$ or $\le m-M'+M<m$. Also,
$\tilde g + D^m a \in I$,
so we obtain a non-trivial $H$-linear combination of elements of $B$
which is equal to zero in $C$. This contradicts to the condition
$\mathrm{Tor}_H C=0$.

Thus we have proved $U$ to be a finite faithful conformal $C$-module
(the action on $v$ is faithful). By Lemma \ref{lem:TorFree}, $U$
is a free $H$-module, so $C$ satisfies (FFR).
\end{proof}

Since (FFR) implies (TU), we obtain the following

\begin{cor}
A finite torsion-free associative conformal algebra
can be embedded into a
finitely generated associative conformal algebra
of no more than linear growth with a two-sided unit.
\end{cor}

\subsection{Infinite associative conformal algebras}

Here we state some examples to show that Theorem \ref{thm:Finite_Ass}
can not be expanded neither to finitely generated associative
conformal algebras of linear growth nor to locally finite
infinitely generated associative conformal algebras.

\begin{exmp}\label{exmp: LinGrowth,noFR}
Consider the free 1-generated $H$-module $M_1=Hv$. Then
\begin{equation}\label{eq:SplitNull}
 C = \Cend_1\oplus M_1
\end{equation}
is an associative conformal
 algebra with respect to the $\lambda $-product
$$
(a+u)\oo{\lambda } (b+w) = a\oo{\lambda } b + a_\lambda (w), \quad
a,b\in \Cend_1, \ u,w\in M_1
$$
(split null extension).
\end{exmp}

By the definition, $\Cend M_1=\Cend_1 \simeq \Bbbk[D,x]$, where
 $x\oo{0} v = Dv$, see \eqref{eq:Cend-Action}.
It is clear that $C$ is a finitely generated associative conformal
algebra of linear growth.

Assume $C$ is embedded into an associative conformal algebra
$C'$ that contains a right unit~$e$.
Then $N_{C'}(v,e)=n>0$ (otherwise, $\{v\oo{0} e\} = 0$),
but
$0=x\oo{0} (v\oo{n} e) = (x\oo{0} v)\oo{n} e = -n v\oo{n-1} e \ne 0$.
The contradiction obtained shows that such $C'$ does not exist,
i.e., $C$ does not meet~(U). Hence, $C^{\mathrm{op}}$ can not be embedded
into an associative conformal algebra with a left unit.

As a corollary, the algebra $C$ from \eqref{eq:SplitNull}
does not satisfy (FR). Indeed,
note that $C$ contains a left unit, so it satisfies~(NA).
Thus, if it has a finite
faithful representation then it also satisfies~(FFR).
But every conformal algebra with (FFR) also satisfies~(TU), which is
not true for~$C$.

Next, let us state an example of a conformal algebra that satisfies
(FR), but does not satisfy~(FFR).

\begin{exmp}\label{exmp:LinGrowth,FR,noFFR}
Assume $C$ is the algebra from Example~\ref{exmp: LinGrowth,noFR}.
Then $C^{\mathrm{op}}$ does not meet (TU), hence, it
does not satisfy~(FFR). Let us show that $C$ has finite faithful
right representation (on an $H$-module with nonzero torsion).
Then $C^{\mathrm{op}}$ would satisfy (FR), but not~(FFR).
\end{exmp}

Consider the right ideal $\Cend_{x,1}$ of $\Cend_1$, and
the free 1-generated $H$-module $M_1$ generated by an element~$v$.
Then
$$
C_1 = \Cend_{x,1}\oplus DM_1 ,
$$
is a right ideal of~$C$.

It is clear that $M=C/C_1$ is a 2-generated $H$-module, the generators
are $\bar 1 = 1+C_2$ and $\bar v = v+ C_2$. Moreover,
$\mathrm {Tor}_H(M) = \Bbbk \bar v\ne 0$.
However, the regular right representation induces
right representation of $C$ on $M$, which is obviously faithful
 ($\bar 1\oo{\lambda } f(D,x) = f(D+\lambda,\lambda )\bar 1$,
$\bar 1\oo{\lambda } f(D) v = f(\lambda )\bar v$).

Finally, let us state an example of a locally finite
torsion-free associative conformal algebra that does not
satisfy neither (U), nor (FR).

\begin{exmp}\label{exmp:LocFinite}
Consider the free 2-generated $H$-module $M_2$ generated by
two elements $e_1$, $e_2$. Identify $M_2$ with $H\otimes \Bbbk ^2$,
assuming $e_1=\begin{pmatrix} 1\\0\end{pmatrix}$,
 $e_2=\begin{pmatrix} 0\\1\end{pmatrix}$.
Then $\Cend M_2$ has natural matrix presentation
as $\mathbb M_2(\Bbbk[D,x])$.

Denote by $C_0$ the set of all matrices
$$
\begin{pmatrix} f(D) & g(D,x) \\ 0 & f(D) \end{pmatrix},
\quad f\in H,\ g\in \Bbbk [D,x].
$$
It follows from \eqref{eq:Cend-Product} that $C_0$
is a subalgebra of $\Cend M_2$.

Now, consider the split null extension
$$
C_{2} = C_0\oplus M_2.
$$
This is a torsion-free infinite associative conformal algebra which is
locally finite, and it has a left unit, i.e., satisfies~(NA).
\end{exmp}

Let us show that $C_{2}$ does not satisfy~(U).
Assume there exists $C'\supseteq C_{2}$ such that $C'$ contains a right
unit~$e$. Denote $n=N_{C'}(e_2,e)$, and note that $n>0$. Also, denote
$$
a_k = \begin{pmatrix} 0 & \frac{x^k}{k!} \\ 0 & 0\end{pmatrix}\in C_0,\quad k\ge 0.
$$
By definition, $a_k\oo{0} e_2 = \frac{1}{k!}D^k e_1$.
Since $N_{C'}(e_1,e)>0$, there exists $m\ge 0$ such that
$e_1\oo{m} e \ne 0$.
Consider
\[
0=a_n\oo{0} (e_2\oo{n+m} e) = (a_n\oo{0} e_2)\oo{n+m} e 
= \frac{1}{n!} D^n e_1 \oo{n+m} e =
(-1)^n\binom{n+m}{n} e_1\oo{m} e \ne 0.
\]
The contradiction obtained shows that $C_{02}$
can not be embedded into an associative conformal algebra with
a right unit. Hence it does not satisfy (FFR) and also~(FR).

\section{Finite solvable Lie conformal algebras}

The main purpose of this section is to show that a finite
torsion-free solvable Lie conformal algebra has a
finite faithful
representation on a free $H$-module.

Throughout this section, $\lambda $-product on a Lie conformal
algebra $L$ is denoted by
$[\cdot\oo{\lambda }\cdot ]:L\otimes L\to L[\lambda ]$.
For $n$-products we use similar notation:
$$
[a\oo{\lambda } b] = \sum\limits_{n=0}^{N_L(a,b)-1}
 \dfrac{\lambda ^n}{n!}[a\oo{n} b], \quad a,b\in L.
$$
If $V$ is a conformal $L$-module with respect to a representation
$\rho: L\to \Cend V$ then $\rho(a)_\lambda v\in V[\lambda ]$ is
simply denoted by $a\oo{\lambda } v$, $a\in L$, $v\in V$. By
$\Kerr V$ we denote the set $\{a\in L\mid a\oo{\lambda } V=0\}$.

\subsection{Double construction for representations}
Let us start with a remark demonstrating the main idea of
this subsection. Consider an ordinary finite-dimensional Lie
algebra $\mathfrak g$, and let $L=\Curr \mathfrak g$ be the
current conformal algebra over $\mathfrak g$. In order to show
that $L$ has a finite faithful (conformal) representation one may
use the classical Ado Theorem: The functor $\Curr $ allows to
raise a representation of $\mathfrak g$ to a representation of
$L$.

However, the same result can be easily obtained as follows.
Consider $\mathfrak g$ as a Leibniz algebra and construct its
finite faithful conformal representation $\rho $
\cite{K.2008b} on the free $H$-module
\[
V=H\otimes (\Bbbk 1\oplus \mathfrak g),
\]
where
\[
\rho(a)_\lambda  1 = \lambda a, \quad \rho(a)_\lambda  b = [a,b],
\quad a,b\in \mathfrak g.
\]
This is straightforward to check that $\rho $ can be extended
 to a conformal representation of~$L$ by
sesqui-linearity.  This
representation is obviously finite and faithful.

A similar idea leads to the following

\begin{thm}\label{thm:DoubleRepr}
Let $L$ be a Lie conformal algebra. Then $L$ has a finite faithful
representation (on a torsion-free module) if and only if there
exist two finite conformal (torsion-free) $L$-modules $V$, $M$,
and a $\Bbbk $-linear map
\[
\langle \cdot \oo{\lambda }\cdot \rangle: L\otimes V\to M[\lambda
]
\]
 such that:
\begin{itemize}
\item[(D1)] $\langle Dx \oo{\lambda } v\rangle
 = -\lambda \langle x\oo{\lambda }v \rangle$,
 $\langle x \oo{\lambda } Dv\rangle
 = (D+\lambda) \langle x\oo{\lambda }v \rangle$
 for all $x\in L$, $v\in V$;
\item[(D2)]
 $x\oo{\lambda } \langle y\oo{\mu } v\rangle
  - \langle y\oo{\mu} (x\oo{\lambda } v)\rangle
  =\langle [x\oo{\lambda } y]\oo{\lambda +\mu }v\rangle $
  for all $x,y\in L$, $v\in V$;
\item[(D3)]
 $\{a\in L\mid \langle a\oo{\lambda } V\rangle=0\}
 \cap\Kerr V\cap \Kerr M=0$.
\end{itemize}
\end{thm}

\begin{proof}
If $L$ has a finite faithful representation on a (torsion-free)
module $U$ then one may consider $V=M=U$ and let
$\langle \cdot \oo{\lambda} \cdot \rangle $
 be the module action of $L$ on
$U$. It is obvious that (D1)--(D3) hold.

Conversely, assume such $V$ and $M$ exist. Consider $U=V\oplus M$
and define an action
$$
(\cdot \hat{\oo{\lambda}} \cdot ): L\otimes U\to U[\lambda ]
$$
as follows:
\begin{equation}\label{eq:MV-Double}
\begin{gathered}
a\hat{\oo{\lambda }} v = a\oo{\lambda } v
  +\lambda \langle a\oo{\lambda } v\rangle , \\
a\hat{\oo{\lambda }} m = a\oo{\lambda } m
\end{gathered}
\end{equation}
for $a\in L$, $v\in V$, $m\in M$. It follows from (D1) that this
map is sesqui-linear. In order to check that \eqref{eq:MV-Double} defines a
conformal representation of $L$ it is enough to show
\begin{equation}\label{eq:DRep1}
x\hat{\oo{\lambda }} (y\hat{\oo{\mu}} v) - y\hat{\oo{\mu }}
(x\hat{\oo{\lambda }} v) =
 [x\oo{\lambda }y] \hat{\oo{\lambda +\mu }} v, \quad x,y\in L, \ v\in V.
\end{equation}
Indeed, (D2) implies
\begin{multline}\nonumber
x\hat{\oo{\lambda }} (y\hat{\oo{\mu}} v) - y\hat{\oo{\mu }}
(x\hat{\oo{\lambda }} v) = x\hat{\oo{\lambda }}
  ( y\oo{\mu } v + \mu \langle y\oo{\mu} v \rangle)
- y\hat{\oo{\mu }} ( x \oo{\lambda } v
 + \lambda \langle x\oo{\lambda } y\rangle )     \\
= x\oo{\lambda }(y\oo{\mu } v) - y\oo{\mu} (x \oo{\lambda } v) +
\lambda \langle x \oo{\lambda } (y\oo{\mu } v) \rangle
- \lambda y\oo{\mu} \langle x\oo{\lambda } y\rangle \\
+ \mu x\oo{\lambda }\langle y\oo{\mu} v \rangle - \mu \langle
y\oo{\mu} (x \oo{\lambda } v)\rangle = [x\oo{\lambda
}y]\oo{\lambda +\mu } v - \lambda \langle [y\oo{\mu} x]\oo{\lambda
+\mu} v \rangle + \mu \langle [x\oo{\lambda } y]\oo{\lambda +\mu}
v \rangle.
\end{multline}
It remains to note that
$$
\langle [y\oo{\mu} x]\oo{\lambda +\mu} v\rangle = -\langle
[x\oo{-D-\mu} y]\oo{\lambda +\mu} v\rangle =  -\langle
[x\oo{\lambda } y]\oo{\lambda +\mu} v\rangle,
$$
and then \eqref{eq:DRep1}  follows.

Finally, note that (D3) guaranties the representation obtained to
be faithful. If both $V$ and $M$ are torsion-free $H$-modules then
so is $U=V\oplus M$.
\end{proof}

\begin{rem}
The representation of $L=\Curr \mathfrak g$ described in the
beginning of this subsection appears from trivial $L$-module $V$ of
rank one and the regular module $M=L$.
\end{rem}

We may now deduce a condition (slightly stronger than just an
embedding of a Lie conformal algebra into an associative one) that
is sufficient for existence of a finite faithful representation.

\begin{cor}\label{cor:Envelope}
Let $L$ be a finite Lie conformal algebra. If for every $0\ne x\in
L$ there exists $(U,\iota )\in \mathcal E(L)$ such that $\iota
(x)\oo{\lambda } U\ne 0$ then $L$ has a finite faithful
representation.
\end{cor}

\begin{proof}
First, let us note that $L$ has to be torsion-free. Indeed,
$\iota(x)\oo{\lambda } U= 0$ for all $x\in \mathrm{Tor}_H L$,
$(U,\iota)\in \mathcal E(L)$. This is impossible if every $x\ne 0$
does not annihilate an appropriate associative envelope.

Next, let us show that there exists $(U,\iota)\in \mathcal E(L)$
such that $\iota(x)\oo{\lambda } U \ne 0$ for all $0\ne x\in L$.
Assume the converse, and choose a basis $B$ of $L$ over $H$.
Consider the family of universal associative envelopes
$(U_N(L), \iota_N)$ \cite{Roitman2000},
where $N$ is a positive integer
considered as an upper bound for the locality function on
$\iota_N(B)\subset U_N(L)$.
Then
$$
I_N = \{x\in L\mid \iota_N(x)\oo{\lambda } U_N(L) = 0\} \ne 0,
\quad N\ge 1,
$$
is a descending chain of ideals of $L$. Moreover, if $h(D)x\in
I_N$ for some $0\ne h\in H$ then $x\in I_N$. Therefore, $L/I_N$ is
a torsion-free $H$-module. Since $L$ is a finite conformal
algebra,
there exists $N_0$ such that $I_N=I_{N_0}$ for all $N\ge N_0$.
 An arbitrary associative envelope $(U, \iota)$
 is an image of $(U_N(L), \iota_N)$ for sufficiently large
$N$, so $\iota(I_{N_0})\oo{\lambda } U = 0$ for all $(U,\iota)\in
\mathcal E(L)$. Hence,
$I_{N_0}=0$.

Finally, choose an envelope $(U,\iota ) \in \mathcal E(L)$ such
that $\iota(x)\oo{\lambda } U \ne 0$ for every $0\ne x\in L$.
Consider $V=\iota(L)\subset U$ as a conformal $L$-module with
respect to the adjoint action: $x\oo{\lambda }\iota(y)=
\iota([x\oo{\lambda }y])$, $x,y\in L$. Denote by $M$ the
$H$-submodule of $U$ spanned by all $\iota(x)\oo{n}\iota(y)$,
$x,y\in L$, $n\ge 0$. Consider $M$ as a conformal $L$-module with
respect to the commutator action:
$$
x\oo{n}(\iota(y)\oo{m}\iota(z)) = [\iota(x)\oo{n} (\iota(y)\oo{m}
\iota(z))],
$$
$x,y,z\in L$, $n,m\ge 0$. Relation \eqref{eq:CommDer} implies that
$x\oo{n}M\subset M$ for all $x\in L$, $n\ge 0$.

Define
$$
\langle \cdot \oo{\lambda } \cdot \rangle:
 L\otimes V \to M[\lambda ]
$$
by the rule
$$
\langle x\oo{\lambda }\iota(y) \rangle = \iota(x)\oo{\lambda
}\iota(y), \quad x,y\in L.
$$

It remains to check that $V$, $M$, and $\langle\cdot \oo{\lambda }
\cdot\rangle$ satisfy the conditions of Theorem~
\ref{thm:DoubleRepr}. Indeed, (D1) holds by definition, (D2)
follows from \eqref{eq:CommDer}, and (D3) is guaranteed by the
choice of $U$: Since $U$ is generated by $\iota(L)$, we have
$\langle x\oo{\lambda } V\rangle =
 \iota(x)\oo{\lambda }\iota(L)\ne 0$ for every $0\ne x\in L$.
\end{proof}

\subsection{Central PBW property}

\begin{prop}
Let $L$ be a finite torsion-free Lie conformal algebra with
a basis $B$ over $H$. Assume $L$ has the PBW property with respect to $B$.
Then $L$ has a finite faithful representation on a free $H$-module.
\end{prop}

\begin{proof}
If for some odd integer $N>0$ the graded universal envelope
$\mathrm{gr}\, U_N(L)$ of $L$
is isomorphic to the free commutative conformal algebra
$\mathfrak F_N(B)$
generated by $B$ with locality
$N$ then $(U_N(L), \iota_N)$
satisfies the conditions of Corollary \ref{cor:Envelope}.

Moreover, it follows from \cite[Theorem 5]{Roitman2000} and the final remark
of \cite[Section~5]{Roitman2000} that for every odd $N>0$
there exists an embedding of conformal algebras
\[
\mathfrak F_N(B) \subseteq
\Curr \Bbbk\big [p_s^b\mid s=0,\dots , \lfloor N/2\rfloor,\, b\in B\big ],
\]
defined by
\[
b\mapsto \sum\limits_{s=0}^{\lfloor N/2\rfloor} \dfrac{1}{s!} (-D)^s\otimes p_s^b,
\quad b\in B.
\]
Therefore, $\mathfrak F_N(B)$ is a torsion-free conformal algebra, and the
modules $V$ and $M$ constructed in the proof of Corollary \ref{cor:Envelope}
are free.
Theorem \ref{thm:DoubleRepr} implies that $L$ has a faithful finite representation
on a free $H$-module.
\end{proof}

\begin{exmp}
There exist finite torsion-free Lie conformal algebras that have no PBW property.
For example, consider the Virasoro conformal algebra
$\mathcal Vir=Hx $ and its trivial module $V=He$ of rank one,
$x\oo{\lambda } e =0$.

Then the split central extension
$L = \mathcal Vir \oplus V$ has no PBW property with respect to $B=\{x,e\}$.
Indeed, for every $(U,\iota )\in \mathcal E(L)$
we have
\[
[\iota(x)\oo{0} (\iota(x)\oo{n} \iota(e))] = -n\iota(x)\oo{n-1} \iota(e), \quad n\ge 1.
\]
Hence, $N_U(\iota(x), \iota(e))=0$.
\end{exmp}

We are going to introduce another property
of the same flavor (called central PBW property) and
prove that it holds for finite solvable Lie conformal algebras.

Let $L$ be a finite torsion-free Lie conformal algebra. Consider
the space $\mathcal L = \Bbbk[t]\otimes_H L$, where $\Bbbk [t]$ is
a right $H$-module with respect to the action defined by
$f(t)\cdot D = -f'(t)$, $f\in \Bbbk[t]$. (This is just the
positive part of the coefficient algebra of $L$
\cite{Kac1997, Roitman1999}.)
Note that $\mathrm{lowdeg}\,(t^m\cdot f(D)) = m-\deg f$, where
$\mathrm{lowdeg}\,(g)$, $g\in \Bbbk[t]$, is the smallest $s$ such that
$t^s$ appears in $g(t)$.

Define the family of linear maps
\[
 (\cdot\oo{n} \cdot): L\otimes \mathcal L \to \mathcal L
\]
by the rule
\begin{equation}\label{eq:CentralAction}
x\oo{n} (t^m\otimes_H a) =
 \sum\limits_{s\ge 0} \binom{n}{s} t^{m+s}\otimes_H [x\oo{n-s} a],
 \quad x,a\in L,\ n,m\ge 0.
\end{equation}

\begin{lem}\label{lem:WellDefined}
The operations \eqref{eq:CentralAction} are well-defined.
\end{lem}

\begin{proof}
For every $x,a\in L$, $n,m\ge 0$ we have
\begin{multline}\nonumber
x\oo{n} (t^m\otimes_H Da) = \sum\limits_{s\ge 0} \binom{n}{s}
t^{m+s} \otimes_H [x\oo{n} Da]
\\
= \sum\limits_{s\ge 0} \binom{n}{s} t^{m+s} \otimes_H
(D[x\oo{n-s} a] + (n-s)[x\oo{n-s-1} a])   \\
= -\sum\limits_{s\ge 0}
 (m+s)\binom{n}{s} t^{m+s-1}\otimes_H [x\oo{n-s} a]
+\sum\limits_{s\ge 0}
 (n-s)\binom{n}{s} t^{m+s} \otimes_H [x\oo{n-s-1} a]  \\
= -m \sum\limits_{s\ge 0} \binom{n}{s} t^{m+s-1}\otimes_H
[x\oo{n-s} a] -\sum\limits_{s\ge 1}
 s\binom{n}{s} t^{m+s-1}\otimes_H [x\oo{n-s} a]  \\
 +
 \sum\limits_{s\ge 0}(s+1)\binom{n}{s+1}
  t^{m+s} \otimes_H [x\oo{n-s-1} a]
= x\oo{n} (t^{m}D\otimes_H a).
\end{multline}
\end{proof}

\begin{lem}\label{lem:CentralActionProp}
The operations \eqref{eq:CentralAction}
  satisfy the following relations:
\[
\begin{gathered}
 Dx \oo{n} u = -n x\oo{n-1} u, \\
 x\oo{n} (y\oo{m} u) - y\oo{m} (x\oo{n} u) =
  \sum\limits_{r\ge 0} \binom{n}{r} [x\oo{n-r} y] \oo{m+r} u,
\end{gathered}
\]
$x,y\in L$, $u\in \mathcal L$, $n,m\ge 0$.
\end{lem}

\begin{proof}
By definition, if $u= t^l\otimes _H a\in \mathcal L$, $a\in L$,
$l\ge 0$, then
\begin{multline}\nonumber
Dx\oo{n} u =
 \sum\limits_{s\ge 0} \binom{n}{s} t^{l+s}\otimes_H [Dx\oo{n-s} a]
 =
 -\sum\limits_{s\ge 0}(n-s) \binom{n}{s} t^{l+s}\otimes_H [x\oo{n-s-1}a]  \\
 =
 -n \sum\limits_{s\ge 0} \binom{n-1}{s} t^{l+s}\otimes _H [x\oo{n-s-1} a]
 =
 -n x\oo{n-1} u
\end{multline}
for all $x\in L$, $n\ge 0$.

Moreover, if $u\in \mathcal L$ as above, $x,y\in L$, $n,m\ge 0$
then
\[
x\oo{n}(y\oo{m} u) = \sum\limits_{p,q\ge 0}
\binom{m}{p}\binom{n}{q} t^{l+p+q}\otimes _H [x\oo{n-q} [y\oo{m-p}
a]].
\]
Similarly,
\[
y\oo{m}(x\oo{n} u) = \sum\limits_{p,q\ge 0}
\binom{m}{p}\binom{n}{q} t^{l+p+q}\otimes _H [y\oo{m-p} [x\oo{n-q}
a]].
\]
Hence,
\begin{multline}\label{eq:CenterJac}
x\oo{n}(y\oo{m} u)-y\oo{m}(x\oo{n} u) \\
= \sum\limits_{p,q\ge 0} \binom{m}{p}\binom{n}{q} t^{l+p+q}\otimes
_H
([x\oo{n-q} [y\oo{m-p} a]] - [y\oo{m-p} [x\oo{n-q} a]])  \\
= \sum\limits_{p,q,s\ge 0} \binom{m}{p}\binom{n}{q}\binom{n-q}{s}
t^{l+p+q}\otimes _H [[x\oo{n-q-s} y]\oo{m-p+s} a] \\
= \sum\limits_{p,r,s\ge 0} \binom{m}{p}\binom{n}{r}\binom{r}{s}
t^{l+p+r-s}\otimes _H [[x\oo{n-r} y]\oo{m-p+s} a]
\end{multline}
for $r=q+s$. Changing the summation variable to $k=p+r-s$ leads
\eqref{eq:CenterJac} to
\begin{multline}\nonumber
\sum\limits_{k,r,s\ge 0} \binom{n}{r}\binom{m}{k-r+s}\binom{r}{s}
t^{l+k}\otimes _H [[x\oo{n-r} y]\oo{m-k+r} a]  \\
= \sum\limits_{r,k\ge 0} \binom{n}{r}\binom{m+r}{k}
 t^{l+k}\otimes _H [[x\oo{n-r} y] \oo{m+r-k} a]
= \sum\limits_{r\ge 0}\binom{n}{r} [x\oo{n-r}y]\oo{m+r} u.
\end{multline}
\end{proof}

Sesqui-linearity \eqref{eq:C3} allows to expand
\eqref{eq:CentralAction} to a family of maps
$(\cdot \oo{n}\cdot ): L\otimes U \to U$,
where $U$ is the free $H$-module generated by
the space~$\mathcal L$. However, $U$ is not a conformal $L$-module
since the locality axiom \eqref{eq:C1} does not hold.

\begin{defn}
Suppose $L$ is a finite torsion-free Lie conformal algebra with a
basis $B$ over $H$. We say that $L$ satisfies {\em central PBW
property\/} with respect to $B$ if there exists a function $N:B\to
\mathbb Z_+$ such that $N(b)>0$ for all $b\in B$ and the subspace
\[
I(B,N) = \Span_{\Bbbk }\{t^m\otimes_H a\mid m\ge N(a), \, a\in B\}
  \subseteq \mathcal L
\]
is invariant under the action \eqref{eq:CentralAction} of $L$.
\end{defn}

\begin{exmp}
If $\mathfrak g$ is a finite dimensional Lie algebra with a basis
$B$ over $\Bbbk $ then $L =\Curr \mathfrak g$ has the central PBW
property with respect to $B$. Indeed, it is enough to consider
$N\equiv 1$

The Virasoro conformal algebra $\mathcal Vir$
 has no central PBW property. If $x$ is the generator of $\mathcal Vir$
 over $H$ then
 $x\oo{0} (t^m\otimes_H x )
 = t^m\otimes_H Dx = -mt^{m-1} \otimes_H x$,
 so $I(B,N)$ can not be invariant under \eqref{eq:CentralAction}
 except for $N\equiv 0$.
\end{exmp}

\begin{thm}\label{thm:ZPBW->FFR}
If $L$ satisfies central PBW property then $L$ has a finite
faithful representation on a torsion-free $H$-module.
\end{thm}

\begin{proof}
Suppose $I(B,N)\subseteq \mathcal L$ is invariant under the action
defined by \eqref{eq:CentralAction}. Let $V=Hu$ be the free
1-generated $H$-module and $M=H\otimes (\mathcal L/I(B,N))$ be the
free $H$-module generated by the finite-dimensional space
$\mathcal L/I(B,N)$.

One may consider $V$ as an $L$-module with respect to the trivial
action $x\oo{\lambda } u =0$, $x\in L$. Note that
\eqref{eq:CentralAction} induces the family of operations $(\cdot
\oo{n}\cdot): L\otimes M \to M$, $n\ge 0$, satisfying
\eqref{eq:C2}, \eqref{eq:C3}, and \eqref{eq:ConfJacobi-n}.
Moreover, these operations also satisfy \eqref{eq:C1} since for
every $w\in \mathcal L$ and for every $x\in L$ we have
$x\oo{n} w \in I(B,N)$
for sufficiently large~$n$. Hence, $M$ is a finite
torsion-free conformal $L$-module.

Define $\langle \cdot\oo{n} \cdot\rangle: L\otimes V \to M$, $n\ge
0$,  in the following way: Set
\begin{equation}\label{eq:Angles}
\langle x\oo{n} u\rangle = 1\otimes(t^n\otimes_H x+I(B,N)), \quad
x\in L,
\end{equation}
and expand to the entire $V$ by the sesqui-linearity (see (D1)
 in Theorem \ref{thm:DoubleRepr}).
Given $x\in L$, the space $I(B,N)$ contains $t^n\otimes_H x$ for
almost all $n\ge 0$, so we may define
$$
\langle x\oo{\lambda } v \rangle =
 \sum\limits_{n\ge 0} \dfrac{\lambda^n}{n!} \langle
  x\oo{n} v\rangle \in M[\lambda ], \quad x\in L,\ v\in V.
$$
It follows from the construction that $\langle Dx\oo{n} v\rangle =
-n\langle x\oo{n-1} v\rangle$, so the condition~(D1)
of Theorem~\ref{thm:DoubleRepr} holds. Later we will identify notations for
$w\in \mathcal L$ and $1\otimes (w+I(B,N))\in M$.

Let us check (D2). It is enough to consider $v=u\in V$. By
\eqref{eq:Angles},
\[
x\oo{n} \langle y\oo{m} u \rangle
- \langle y\oo{m} (x\oo{n} u) \rangle
= \sum\limits_{s\ge 0}\binom{n}{s} t^{m+s} \otimes_H [x\oo{n-s} y]
= \sum\limits_{s\ge 0} \binom{n}{s}
 \langle [x\oo{n-s} y] \oo{m+s} u\rangle,
\]
for all $x,y\in L$, $n,m\ge 0$. This implies (D2).

Finally, assume that $\langle x\oo{\lambda } u\rangle =0$ for some
nonzero $x=\sum_{i} h_i(D)b_i \in L$, $h_i\in H$, $b_i\in B$. Then
$x$ can be rewritten in the form
\[
x= \sum\limits_{s=n}^m D^s y_s, \quad y_s\in \Bbbk B, \ y_n,y_m\ne 0.
\]
Therefore, $0=\langle x\oo{n} u\rangle = t^n\otimes_H D^ny_n =
 (-1)^nn! 1\otimes _H y_n$, but $1\otimes_H y_n\notin I(B,N)$
 if $0\ne y_n\in \Bbbk B$ and $N(b)>0$ for all $b\in B$.
Hence, the condition (D3) of Theorem \ref{thm:DoubleRepr} holds,
so $L$ has finite faithful representation on the free $H$-module
$V\oplus M$.
\end{proof}

\subsection{Representations of solvable algebras}

Suppose $L$ is a Lie conformal algebra. Denote by $L'$ the
subspace of $L$ spanned by all elements $[x\oo{n} y]$, $x,y\in L$,
$n\ge 0$. It is clear that $L'$ is an ideal of~$L$. Define the
chain of ideals $L^{(n)}$ as follows:
\[
L^1=L,\quad L^{(n+1)}= (L^{(n)})',\ n\ge 1.
\]
If there exists $n> 1$ such that $L^{(n)}=0$ then $L$ is said to
be solvable.

The structure of finite solvable Lie conformal algebras was
described in \cite[Theorems 8.4, 8.5]{DK1998}.

\begin{thm}[Conformal version of Lie's Theorem,
               \cite{DK1998}]\label{thm:DK_Solvable}
Let $L$ be a finite torsion-free solvable Lie conformal algebra,
and assume the base field $\Bbbk $ is algebraically closed. Then
there exists a sequence of ideals
\[
0=L_0\subset L_1\subset \dots \subset L_n = L
\]
such that $L_{i+1}/L_i$ is a free $H$-module of rank one for each
$i=0,\dots ,n-1$. Moreover, the adjoint representation of $L$ on
itself induces the structure of a weight conformal $L$-module on
$L_{i+1}/L_i = H\bar a$, i.e., there exists a linear map $\varphi:
L\to \Bbbk[\lambda ]$, $\varphi: x\mapsto \varphi_x(\lambda )$,
such that
\[
[x\oo{\lambda} \bar a] = \varphi_x(\lambda )\bar a.
\]
\end{thm}

\begin{thm}\label{thm:ZPBW_Solvable}
For every finite torsion-free solvable Lie conformal algebra $L$
over an algebraically closed
field $\Bbbk $ there exists a basis $B$ of $L$ over $H$ such that
$L$ satisfies the central PBW-property with respect to $B$.
\end{thm}

\begin{proof}
Theorem \ref{thm:DK_Solvable} implies that there exists
a basis $B=\{b_1,\dots, b_n\}$ of $L$ such that
\[
[a\oo{\lambda } b_i] = \sum\limits_{j=i}^n
  \varphi_{a,i}^j(D,\lambda )b_j, \quad
   a\in L,\ \varphi_{a,i}^j\in \Bbbk[D,\lambda ],\ i=1,\dots, n,
\]
i.e., the regular representation of $L$ on itself
is triangular. Moreover,
$\varphi_{a,i}^i(D,\lambda ) = \varphi_a(\lambda )$.

Define a function $N:B\to \mathbb Z_+$ in the following way.
Choose any integer $K>0$, and set
\[
N(b_n) = K,\quad N(b_i) = \max\limits_{j>i,a\in B}
 \{N(b_j) + \deg_D \varphi_{a,i}^j \},\ 1\le i\le n-1.
\]
Then
$I(B,N)$ is invariant with respect to \eqref{eq:CentralAction}.
Indeed, let $a\in A$ and
\[
 \varphi_{a,i}^j = \sum\limits_{s\ge 0}
  \dfrac{\lambda ^s}{s!} f_{a,i}^{j,s}(D), \quad f_{a,i}^{j,s}\in H,
  \ 1\le i\le j\le n.
\]
In particular, $N(b_i)\ge N(b_j)+\deg f_{a,i}^{j,s}$ for all $j>i$.
Then for every $m>N(b_i) $ we have
\begin{multline}\nonumber
a\oo{k} (t^m\otimes_H b_i)
 = \sum\limits_{s\ge 0} \binom{k}{s} t^{m+s}\otimes_H
  \left(\sum\limits_{j=i}^n f_{a,i}^{j,k-s}(D) b_j \right)  \\
 =\sum\limits_{j=i}^n
  \bigg(\sum\limits_{s\ge 0}
  \binom{k}{s} t^{m+s}\cdot f_{a,i}^{j,k-s}(D)
   \bigg)\otimes_H b_j
 =\sum\limits_{j=i}^n g_{a,i}^{k,j}(t)\otimes_H b_j,
\end{multline}
where
$\mathrm{lowdeg}\, g_{a,i}^{k,j} \ge
 m - \max\limits_{s\ge 0}\deg f_{a,i}^{j,k-s} \ge N(b_j)$
 for $j>i$. If $j=i$ then $f_{a,i}^{i,k-s}\in \Bbbk $, so, finally,
\[
a\oo{k} (t^m\otimes_H b_i) \in I(B,N), \quad a\in B,
\ m\ge N(b_i),\ i=1,\dots, n.
\]
This is enough to conclude that $L$ has the central PBW property
with respect to~$B$.
\end{proof}

\begin{thm}\label{thm:FiniteReprExists}
A finite torsion-free solvable Lie
conformal algebra $L$
has a finite faithful representation
on a free $H$-module.
\end{thm}

\begin{proof}
If the field $\Bbbk $ is algebraically closed
then the claim follows from Theorems \ref{thm:ZPBW->FFR}
and \ref{thm:ZPBW_Solvable}.

Suppose $\Bbbk $ is an arbitrary field (of characteristic zero),
and let $\bar\Bbbk$ stands for its algebraic closure.
Then $\bar L = \bar\Bbbk\otimes_\Bbbk L$ is a Lie conformal algebra
over $\bar\Bbbk$. If $L$ is finite (torsion-free, solvable)
then so is $\bar L$.

To complete the proof, we only need the following

\begin{lem}
Suppose $L$ is a finite torsion-free
Lie conformal algebra over $\Bbbk $.
If $\bar L = \bar\Bbbk\otimes_\Bbbk L$
has a finite faithful representation on a free module
over $\bar H= \bar\Bbbk[D]$ then $L$ has a finite faithful
representation on a free $H$-module.
\end{lem}

\begin{proof}
Let $B$ stands for a finite basis of $L$ over $H=\Bbbk[D]$,
then $B$ is a basis of $\bar L$ over $\bar H=\bar\Bbbk [D]$.
The multiplication  table of $\bar L$ with respect to $B$
coincides with the one of $L$. Assume
\[
[a\oo{\lambda } b] = \sum\limits_{c\in B} g_{a,b}^c(D,\lambda )c,
\quad a,b\in B,\ g_{a,b}^c\in \Bbbk [D,\lambda ].
\]

Suppose $\bar V$ is a finite faithful conformal $\bar L$-module
which is free as an $\bar H$-module.
Let us choose a basis $\{e_1,\dots, e_m \}$
 of $\bar V$ over $\bar H$. Then the
structure of a conformal $\bar L$-module on $\bar V$ is given
by
\[
 a\oo{\lambda } e_i = \sum\limits_{j=1}^m \varphi_{a,i}^j (D,\lambda )e_j,
\quad i=1,\dots, m,\ a\in B,\ \varphi_{a,i}^j\in \bar\Bbbk[D,\lambda ].
\]

Since the number of polynomials $\varphi_{a,i}^j$ is finite,
there exists a finite extension $F$ of the field $\Bbbk $
such that all these polynomials belong to $F[D,\lambda ]$.
We may assume $F=\Bbbk(\alpha )$, where $\alpha $ is algebraic
over~$\Bbbk $. Suppose $n\ge 1$ is the degree of the minimal polynomial
for~$\alpha $.
Then we may find
$f_{b,i,k}^{j,l}\in \Bbbk[D,\lambda ]$,
$k=0,\dots, n-1$, $i,j=1,\dots, m$,  $b\in B$,
such that
\begin{equation}\label{eq:Varphi-f}
\alpha^k \varphi_{b,i}^j(D,\lambda ) =
  \sum\limits_{l=0}^{n-1} \alpha ^l f_{b,i,k}^{j,l}(D,\lambda ).
\end{equation}

Consider the free $H$-module $W$ generated by the finite
set $\{e_i^k\mid i=1,\dots, m,\, k=0,\dots, n-1 \}$
and define
\begin{equation}\label{eq:Ext-Action}
b\oo{\lambda } e_i^k
 = \sum\limits_{j=1}^m\sum\limits_{l=0}^{n-1}
 f_{b,i,k}^{j,l}(D,\lambda )e_j^l,
 \quad
b\in B.
\end{equation}
Relation \eqref{eq:Ext-Action} can be extended by sesqui-linearity
to a map $(\cdot\oo{\lambda } \cdot): L\otimes W \to W[\lambda ]$.
Let us check that the analogue of \eqref{eq:ConfJacobi}
holds for this operation, i.e., in our notations, we have to show
the equation
\begin{multline}\label{eq:Ext-Jac2}
\sum\limits_{j,p=1}^m \sum\limits_{l,q=0}^{n-1}
 \big [
f_{b,i,k}^{j,l} (D+\lambda )f_{a,j,l}^{p,q}(D,\lambda )
 - f_{a,i,k}^{j,l}(D+\mu, \lambda ) f_{b,j,l}^{p,q}(D,\mu )
 \big] \\
=
 \sum\limits_{c\in B}\sum\limits_{p=1}^m\sum\limits_{q=0}^{n-1}
  g_{a,b}^c(-\lambda -\mu, \lambda )f_{c,i,k}^{p,q}(D,\lambda +\mu)
\end{multline}
for all $a,b\in B$, $i=1,\dots, m$, $k=0,\dots, n-1$.

It follows from \eqref{eq:ConfJacobi} for the action of
$\bar L$ on $\bar V$ that
\begin{multline}\label{eq:ExtJac-1}
\sum\limits_{j,p=1}^m
\big [\varphi_{b,i}^j(D+\lambda ,\mu)\varphi_{a,j}^p(D,\lambda )
  - \varphi_{a,i}^j (D+\mu)\varphi_{b,j}^p (D,\mu ) \big] \\
  =
\sum\limits_{c\in B} \sum\limits_{p=1}^m
  g_{a,b}^c(-\lambda -\mu, \lambda )\varphi_{c,i}^p(D,\lambda+\mu ),
  \quad
a,b\in B,\ i=1,\dots, m.
\end{multline}
Now we can multiply both sides of \eqref{eq:ExtJac-1} by $\alpha^k$,
$k=0,\dots, n-1$, and apply \eqref{eq:Varphi-f} to obtain
\eqref{eq:Ext-Jac2}.

Finally, if $x=\sum\limits_{b\in B} f_b(D)b \in \Kerr W$ for some
$f_b(D)\in H$ then we have
\begin{equation}\label{eq:ExtKernel}
\sum\limits_{b\in B} f_b(-\lambda ) f_{b,i,k}^{j,l}(D,\lambda )=0
\end{equation}
for all $j=1,\dots, m$, $l=0,\dots, n-1$. In particular,
\[
\sum\limits_{l=0}^{n-1}\alpha^l \left (
\sum\limits_{b\in B} f_b(-\lambda ) f_{b,i,k}^{j,l}(D,\lambda )
\right ) =
\alpha^k \sum\limits_{b\in B} f_b(-\lambda )
  \varphi_{b,i}^j (D,\lambda ) =0.
\]
Therefore, $x\oo{\lambda }\bar V = 0$. Since $\bar V$ is faithful, we have
$x=0$, so $W$ is a faithful conformal $L$-module.
\end{proof}

\end{proof}

\subsection*{Acknowledgements}
The author gratefully acknowledges the support from Lavrent'ev
Young Scientists Competition (No 43 on 04.02.2010) and
 SSc-3669.2010.1.
The major part of the work was done in the
University of California, San Diego. The author thanks the
UCSD Department of Mathematics for the hospitality during the visit.

\end{document}